\documentclass[11pt]{amsart}
\usepackage{a4wide,color}

\usepackage{amsfonts, amsmath, amscd}
\usepackage[psamsfonts]{amssymb}
\usepackage{amssymb,textcomp,lmodern}
\usepackage{pb-diagram}
\usepackage[all,cmtip]{xy}

\newtheorem{theorem}{Theorem}[section]
\newtheorem{lemma}[theorem]{Lemma}
\newtheorem{corollary}[theorem]{Corollary}

\newtheorem{remark}[theorem]{Remark}
\newtheorem{proposition}[theorem]{Proposition}

\newtheorem{example}[theorem]{Example}

\theoremstyle{definition}
\newtheorem{definition}[theorem]{Definition}

\numberwithin{equation}{section}

\newcommand{\e}{\varepsilon}
\newcommand{\w}{\omega}

\newcommand{\IR}{\mathbb{R}}
\newcommand{\IN}{\mathbb N}

\newcommand{\ff}{\mathbb{F}}

\newcommand{\IF}{\mathbb{F}}

\newcommand{\TTT}{\mathcal{T}}
\newcommand{\FF}{\mathcal{F}}
\newcommand{\F}{\mathcal{F}}

\newcommand{\M}{\mathcal{M}}

\newcommand{\supp}{\mathrm{supp}}

\newcommand{\spn}{\mathrm{span}}

\newcommand{\CC}{C_k}

\newcommand{\SM}{{\setminus}}

\input xy
\xyoption{all}

\title[The Josefson--Nissenzweig property for locally convex spaces]{The Josefson--Nissenzweig property\\ for locally convex spaces}
\author{Taras Banakh and Saak Gabriyelyan}

\address{Ivan Franko National University of Lviv (Ukraine) and Jan Kochanowski University in Kielce (Poland)}
\email{t.o.banakh@gmail.com}

\address{Department of Mathematics, Ben-Gurion University of the Negev, Beer-Sheva, P.O. 653, Israel}
\email{saak@math.bgu.ac.il}



\begin{document}

\begin{abstract}
We define a locally convex space $E$ to have the {\em Josefson--Nissenzweig property} (JNP) if the identity map $(E',\sigma(E',E))\to ( E',\beta^\ast(E',E))$ is not sequentially continuous.
By the classical Josefson--Nissenzweig theorem, every infinite-dimensional Banach space has the JNP. A characterization of locally convex spaces with the JNP is given. We thoroughly study the JNP in various function spaces. Among other results we show that for a Tychonoff space $X$, the function space $C_p(X)$ has the JNP iff there is a weak$^\ast$ null-sequence $(\mu_n)_{n\in\w}$  of finitely supported sign-measures on $X$ with unit norm. However, for every Tychonoff space $X$, neither the space $B_1(X)$ of Baire-1 functions on $X$ nor the free locally convex space $L(X)$ over $X$ has the JNP.
\end{abstract}

\maketitle


\section{Introduction}


All locally convex spaces (lcs for short) are assumed to be Hausdorff and infinite-dimensional, and all topological spaces are assumed to be infinite and Tychonoff. We denote by $E'$ the topological dual of an lcs $E$. The dual space $E'$ of $E$ endowed with the weak$^\ast$ topology $\sigma(E',E)$ and the strong topology $\beta(E',E)$ is denoted by $E'_{w^\ast}$ and $E'_\beta$, respectively.  For a bounded subset $B\subseteq E$ and a functional $\chi\in E'$, we put $\|\chi\|_B:= \sup\{ |\chi(x)|:x\in B\cup\{0\}\}$.

Josefson \cite{Josefson} and Nissenzweig \cite{Nissen} proved independently the following theorem (other proofs of this beautiful result  were given by Hagler and Johnson \cite{HagJohn} and Bourgain and Diestel  \cite{BourDies}).
\begin{theorem}[Josefson--Nissenzweig] \label{t:JN-theorem}
Let $E$ be a Banach space. Then there is a null sequence $\{\chi_n\}_{n\in\w}$ in $E'_{w^\ast}$ such that $\| \chi_n\|=1$ for every $n\in\w$.
\end{theorem}
\noindent Therefore the identity map $E'_{w^\ast} \to E'_\beta$ is {\em not} sequentially continuous for every Banach space. Recall that a function $f:X\to Y$ between topological spaces $X$ and $Y$ is called {\em sequentially continuous} if for any convergent sequence $\{x_n\}_{n\in\w}\subseteq X$, the sequence $\{f(x_n)\}_{n\in\w}$ converges in $Y$ and $\lim_{n}f(x_n)=f(\lim_{n}x_n)$.

The Josefson--Nissenzweig theorem was extended to Fr\'{e}chet spaces in \cite{BLV}.
\begin{theorem}[Bonet--Lindstr\"{o}m--Valdivia] \label{t:BLV-theorem}
For a Fr\'{e}chet space $E$, the identity map $E'_{w^\ast} \to E'_\beta$ is  sequentially continuous if and only if $E$ is a Montel space.
\end{theorem}


\noindent Another extension of  the Josefson--Nissenzweig theorem was provided by Bonet \cite{Bonet-JN} and Lindstr\"{o}m and  Schlumprecht \cite{Lind-Schlump} who proved that a Fr\'{e}chet space $E$ is a Schwartz space if and only if every  null sequence in $E'_{w^\ast}$ converges uniformly to zero on some zero-neighborhood in $E$.

Studying the separable quotient problem for $C_p$-spaces and being motivated by the Josefson--Nissenzweig theorem, Banakh, K\c akol and \'Sliwa  introduced in \cite{BKS} the Josefson--Nissenzweig property for the space $C_p(X)$ of continuous real-valued functions on a Tychonoff space $X$, endowed with the topology of pointwise convergence. Namely, they defined $C_p(X)$ to have the Josefson--Nissenzweig property (JNP) if the dual space $C_p(X)'$ of $C_p(X)$ contains a weak$^\ast$ null sequence $\{\mu_n\}_{n\in\w}$  of finitely supported sign-measures on $X$ such that $\|\mu_n\|:=|\mu_n|(X)=1$ for every $n\in\w$. This definition implies that for any Tychonoff space $X$ containing a non-trivial convergent sequence, the function space $C_p(X)$ has the JNP. Yet, there exists a compact space $K$ without non-trivial convergent sequences such that the function space $C_p(K)$ has the Josefson--Nissenzweig property, see Plebanek's example in \cite{BKS}. On the other hand, Banakh, K\c akol and \'Sliwa observed in \cite{BKS} that the function space $C_p(\beta\w)$ does not have the Josefson--Nissenzweig property.

Denote by $C^0_p(\w)$ the subspace of the product $\IR^\w$ consisting of all real-valued  functions on the discrete space $\w$ that tend to zero at infinity. The following characterization of $C_p$-spaces with the Josefson--Nissenzweig property is the main result of \cite{BKS}.

\begin{theorem}[Banakh--K\c akol--\'Sliwa] \label{t:JPN-Cp-BKS}
For a Tychonoff space  $X$, the following conditions are equivalent:
\begin{enumerate}
\item[{\rm (i)}]  $C_p(X)$ has the Josefson--Nissenzweig property;
\item[{\rm (ii)}]  $C_p(X)$ contains a complemented subspace isomorphic to $C^0_p(\w)$;
\item[{\rm (iii)}]  $C_p(X)$ has a quotient isomorphic to $C^0_p(\w)$;
\item[{\rm (iv)}]  $C_p(X)$ admits a linear continuous map onto $C^0_p(\w)$.
\end{enumerate}
\end{theorem}

\noindent This characterization shows that the Josefson--Nissenzweig property depends only on the locally convex structure of the space $C_p(X)$ in spite of the fact that its definition involves the norm in the dual space, which is a kind of an external structure for $C_p(X)$.

The aforementioned results motivate us  to define and study the Jossefson--Nissenzweig property in the class of all locally convex spaces, this is the main goal of the article.

Let $E$ be a locally convex space. Taking into account Theorem \ref{t:BLV-theorem}, it is natural  to say that $E$ has the Jossefson--Nissenzweig property if the identity map $E'_{w^\ast} \to E'_\beta$ is  not sequentially continuous. However, such definition is not fully consistent with the Josefson--Nissenzweig property for $C_p$-spaces:  the identity map $C_p(X)'_{w^\ast}\to C_p(X)'_\beta$ is not sequentially continuous for any infinite Tychonoff space $X$ containing a non-trivial convergent sequence $\{x_n\}_{n\in\w}$ because the strong dual of $C_p(X)$ is feral (see for example \cite[Proposition~2.6]{GK-DP}, recall that a locally convex space $E$ is called {\em feral} if any bounded subset of $E$ is finite-dimensional). Indeed, the sequence $\{\tfrac{1}{n}\delta_{x_n}\}_{n\in\w}$, where $\delta_x$ is the Dirac measure at the point $x$, is trivially $w^\ast$-null but it is not $\beta$-null because of ferality.  Therefore to give a ``right'' definition of the Josefson--Nissenzweig property, we should consider a weaker topology on $E'$ than the strong topology $\beta(E',E)$. A locally convex topology on the dual $E'$ which covers both cases described in Theorem \ref{t:BLV-theorem} and Theorem \ref{t:JPN-Cp-BKS} is the topology $\beta^\ast(E',E)$ of uniform convergence on $\beta(E,E')$-bounded subsets of $E$,  defined in \cite[8.4.3.C]{Jar}. Put $E'_{\beta^\ast}:= \big(E', \beta^\ast(E',E)\big)$.

\begin{definition} \label{def:JNP}
A locally convex space $E$ is said to have the {\em Jossefson--Nissenzweig property} (briefly, the JNP) if the identity map $E'_{w^\ast} \to E'_{\beta^\ast}$ is  not sequentially continuous.
\end{definition}

\noindent Now, if $E$ is a Fr\'{e}chet space, then $\beta^\ast(E',E)=\beta(E',E)$ by Corollary 10.2.2 of \cite{Jar}. Therefore Theorem \ref{t:BLV-theorem} states that a Fr\'{e}chet space $E$ has the JNP if and only if $E$ is not a Montel space.
In Section~\ref{s:JNP} we give an independent proof of Theorem \ref{t:BLV-theorem} in an extended form, see Theorem \ref{t:BLV-theorem-JNP}.
We also consider other important classes of locally convex spaces. In particular, we prove that for any Tychonoff space $X$, the space $B_1(X)$ of all Baire one functions on $X$ and the free locally convex space $L(X)$ over $X$ fail to have the JNP. In Theorem \ref{t:JNP-characterization} we give an operator characterization of locally convex spaces with the JNP. 
In Section \ref{sec:JNP-function} we study function spaces endowed with various natural topologies which have the JNP. In particular, in Corollary \ref{c:JNP-Cp} we show  that a function space $C_p(X)$ has the JNP if and only if it has that property in the sense of Banakh, K\c akol and \'Sliwa \cite{BKS} mentioned above. 


\section{The Josefson--Nissenzweig property in some classes of locally convex spaces} \label{s:JNP}


All locally convex spaces considered in this paper are  over the field $\IF$ of real or complex numbers.  

Let $E$ be a locally convex space. A closed absorbent absolutely convex subset of $E$ is called a {\em barrel}.
The {\em polar} of a subset $A$ of $E$ is denoted by
\[
A^\circ :=\{ \chi\in E': \|\chi\|_A \leq 1\}, \quad\mbox{ where }\quad\|\chi\|_A=\sup\big\{|\chi(x)|: x\in A\cup\{0\}\big\}.
\]

Now we give a more clear and useful description of the topology $\beta^\ast(E',E)$. Recall that we defined a subset $B\subseteq E$ to be {\em barrel-bounded} if for any barrel $U\subseteq E$ there is an $n\in\w$ such that $B\subseteq nU$.
It is easy to see that each finite subset of $E$ is barrel-bounded and each barrel-bounded set in $E$ is bounded.
Observe that a subset of a barrelled space is bounded if and only if it is barrel-bounded. We recall that a locally convex space  $E$ is {\em barrelled} if each barrel in $E$ is a neighborhood of zero.
A neighborhood base at zero of the topology $\beta^\ast(E',E)$ on $E'$ consists of the polars $B^\circ$  of barrel-bounded subsets $B\subseteq E$.

Let $E$ be a locally convex space. The space $E$ with the weak topology $\sigma(E,E')$ is denoted by $E_w$. The strong second dual space $(E'_\beta)'_\beta$ of $E$ will be denoted by $E''$. Denote by $\psi_E: E\to E''$ the canonical evaluation map defined by $\psi_E(x)(\chi):=\chi(x)$ for all $x\in E$ and $\chi\in E'$. Recall that $E$ is called {\em reflexive} if $\psi_E$ is a topological isomorphism, and $E$ is  {\em Montel} if it is reflexive and every closed bounded subset of $E$ is compact. We recall that $E$ has the {\em  Shur property } if the identity map $E_w\to E$ is sequentially continuous. Recall also that $E$ is called {\em quasi-complete} if every closed bounded subset of $E$ is complete.

\begin{theorem} \label{t:JNP-Montel}
Let $E$ be a quasi-complete reflexive space whose every separable bounded subset is metrizable. Then $E$ has the JNP if and only if $E$ is not Montel.
\end{theorem}

\begin{proof}
Recall that each reflexive space is barrelled, see \cite[Proposition~11.4.2]{Jar}. Now we note that $\beta^\ast(E',E)=\beta(E',E)$ by Corollary 10.2.2 of \cite{Jar}.

To prove the ``only if'' part, suppose for a contradiction that $E$ is Montel. Since the strong dual of a Montel space is also Montel (\cite[Proposition~11.5.4]{Jar}), the strong dual space $E'_\beta$  is Montel and, by Proposition 2.3 of \cite{Gabr-free-resp}, $E'_\beta$ has the Schur property. As every Montel space is also reflexive, the Schur property exactly means that the identity map $E'_{w^\ast}\to E'_\beta=E'_{\beta^*}$ is sequentially continuous and hence $E$ does not have the JNP. This contradiction shows that $E$ is not Montel.
\smallskip

To prove the ``if'' part, assume that $E$ is not Montel. Then, by Proposition 3.7 of \cite{Gabr-free-resp}, $E'$ contains a $\sigma(E',E)$-convergent sequence which does not converge in $ \beta(E',E)=\beta^\ast(E',E)$. Thus $E$ has the JNP.
\end{proof}

\begin{proposition} \label{p:bar-weak-JNP}
Let $E$ be a barrelled space, and let $\TTT$ be a locally convex topology on $E$ compatible with the duality $(E,E')$. Then the space $E$ has the JNP if and only if $(E,\TTT)$ has the JNP.
\end{proposition}

\begin{proof}
Note that $(E,\TTT)'=E'$ and hence $(E,\TTT)'_{w^\ast} =E'_{w^\ast}$. So to prove the proposition it suffices to show that also $(E,\TTT)'_{\beta^\ast} =E'_{\beta^\ast}$. To this end, we shall show that $\beta^\ast \big(E',(E,\TTT)\big)=\beta(E',E)=\beta^\ast(E',E)$. It is clear that these equalities hold true if the spaces $E$ and $(E,\TTT)$ have the same barrel-bounded sets. Since $E$ is barrelled,  a subset of $E$ is barrel-bounded if and only if it is bounded. As $E$ and $(E,\TTT)$ have the same bounded sets, these spaces have the same  barrel-bounded sets if we shall show  that every bounded subset $B$ of $(E,\TTT)$ is barrel-bounded. If $U$ is a barrel in $(E,\TTT)$, it is also a barrel in $E$ and hence $U$ is a neighborhood of zero in $E$. Hence there is $n\in\w$ such that $B\subseteq nU$. Thus $B$ is barrel-bounded.
\end{proof}

Below we give an independent proof of Theorem \ref{t:BLV-theorem}. Our proof, as well as the proof of Theorem \ref{t:BLV-theorem}, also essentially uses the following result from \cite{Lind-Schlump}:  {\em a Fr\'{e}chet space $E$ is reflexive if the identity map $E'_{{w^\ast}}\to E'_\beta$ is sequentially continuous}.

\begin{theorem} \label{t:BLV-theorem-JNP}
For a Fr\'{e}chet space $E$, the following assertions are equivalent:
\begin{enumerate}
\item[{\rm (i)}] $E$ is not  Montel.
\item[{\rm (ii)}] $E$ has the JNP.
\item[{\rm (iii)}] $E_w$ has the JNP.
\end{enumerate}
\end{theorem}

\begin{proof}
(i)$\Rightarrow$(ii) Let $E$ be a non-Montel space. Since the Fr\'echet space $E$ is barelled, $E'_{\beta}=E'_{\beta^*}$. If the identity map $E'_{w^\ast}\to E'_\beta=E'_{\beta^*}$ is not sequentially continuous, then $E$ has the JNP by definition. If this map is sequentially continuous, then, by \cite{Lind-Schlump}, $E$ is reflexive. Finally, Theorem \ref{t:JNP-Montel} implies that $E$ has the JNP.

(ii)$\Rightarrow$(i)  Assume that $E$ has the JNP. Then $E$ is not Montel by Theorem \ref{t:JNP-Montel} (recall that every Montel space is reflexive).

(iii)$\Leftrightarrow$(ii) follows from Proposition \ref{p:bar-weak-JNP}. 
\end{proof}



Now we present a characterization of the JNP in the terms of $C_p^0(\w)$-valued operators.
An operator $T:X\to Y$ between locally convex spaces is {\em $\beta$-to-$\beta$ precompact} if for any barrel-bounded set $B\subseteq X$ the image $T(B)$ is barrel-precompact in $Y$. {Recall that $C_p^0(\w)$  denotes the subspace of $\IF^\w$ consisting of functions $\w\to\IF$ that tend to zero at infinity. We shall use also the following well known description of precompact subsets of the Banach space $c_0$, where $e'_n$ is the $n$th coordinate functional of $c_0$.

\begin{proposition} \label{p:precompact-c0}
A subset $A$ of $c_0$ is precompact if and only if $\lim\limits_{n\to\infty}\|e'_n\|_A=0$.
\end{proposition}

Now we are ready to prove an operator characterization of the JNP.
\begin{theorem} \label{t:JNP-characterization}
For a locally convex space $E$ over the field\/ $\IF$ the following conditions are equivalent:
\begin{enumerate}
\item[{\rm (i)}] $E$ has the Josefson-Nissenzweig property;
\item[{\rm (ii)}]  there exists a continuous operator $T:E\to C_p^0(\w)$, which is not $\beta$-to-$\beta$ precompact.
\end{enumerate}
\end{theorem}

\begin{proof}
(i)$\Rightarrow$(ii) Assume that $E$ has the JNP, and take any null sequence $\{\mu_n\}_{n\in\w}\subseteq E'_{w^\ast}$ that does not converge to zero in the topology $\beta^\ast(E',E)$. Then there exist a barrel-bounded set $B\subseteq E$ and $\e>0$ such that the set $\{n\in\w: \|\mu_n\|_B >\e\}$ is infinite. Replacing $\{\mu_n\}_{n\in\w}$ by a suitable subsequence and multiplying it by $\tfrac{1}{\e}$, we can assume that $ \|\mu_n\|_B >1$ for every $n\in\w$. The $\sigma(E',E)$-null sequence $\{\mu_n\}_{n\in\w}$ determines the continuous operator $T:X\to C_p^0(\w)$ defined by  $T(x):= (\mu_n(x))_{n\in\w}$. For simplicity of notation, set $Y:=C_p^0(\w)$. By Lemma 2.4 of \cite{BG-sGP}, 
the strong topology $\beta(Y,Y')$ on $Y$ is generated by the norm $\|(x_n)_{n\in\w}\|=\sup_{n\in\w}|x_n|$  (so $\big(Y,\beta(Y,Y')\big)=c_0$). To see that the set $T(B)$ is not precompact in the norm topology, for every $k\in\w$ we can inductively choose a number $n_k\in\w$ and a point $b_k\in B$ such that the following conditions are satisfied:
\begin{itemize}
\item $|\mu_n(b_i)|<\frac{1}{2}$ for any $i<k$ and any $n\ge n_k$;
\item $|\mu_{n_k}(b_k)|>1$;
\end{itemize}
(this is possible because $\lim_n \mu_n(b)=0$ and $\|\mu_n\|_B >1$).
Then for any $i<k$, we have
\[
\|T(b_i)-T(b_k)\|\ge |\mu_{n_k}(b_k) -\mu_{n_k}(b_i)|>1-\tfrac{1}{2}>\tfrac{1}{2},
\]
which implies that the sequence $\{T(b_k)\}_{k\in\w}$ has no accumulation point in the Banach space $c_0$, and hence it cannot be precompact in $c_0$.
\smallskip

(ii)$\Rightarrow$(i) Assume that $E$ admits a continuous operator $T:E\to C_p^0(\w)$ such that $T$ is not $\beta$-to-$\beta$ precompact. Then for some barrel-bounded set $B\subseteq X$, the image $T(B)$ is not barrel-precompact in $C_p^0(\w)$.

Observe that any barrel $B$ in $C_p^0(\w)$ is also a barrel in the Banach space $c_0$, and hence $B$ is a $c_0$-neighborhood of zero. This implies that a subset $A$ of $C_p^0(\w)$ is barrel-bounded (resp. barrel-precompact) if and only if it is a bounded (resp. precompact) subset in $c_0$.
Therefore $T(B)$ is not precompact in $c_0$. By Proposition \ref{p:precompact-c0} this means that $\|e'_n\|_{T(B)}\not\to 0$. For every $n\in\w$, set $\chi_n:=e'_n\circ T$. Since $e'_n\to 0$ in the weak$^\ast$ topology of $C_p^0(\w)'$, we obtain that $\chi_n\to 0$  in the weak$^\ast$ topology of $E'$. Since
\[
\| \chi_n\|_B =  \|e'_n\|_{T(B)} \not\to 0 \quad \mbox{ as } \; n\to\infty,
\]
the sequence $\{ \chi_n\}_{n\in\w}$ does not converge to zero in the topology $\beta^\ast(E',E)$. Thus $E$ has the JNP.
\end{proof}

Recall that a locally convex space $E$ is {\em $c_0$-barrelled\/} if every $\sigma(E',E)$-null sequence is equicontinuous.  It is well known that each Fr\'echet space is barrelled and each barrelled locally convex space is $c_0$-barrelled. 

\begin{theorem} \label{t:JNP-c0-barreled}
Let $E$ be a $c_0$-barrelled space such that $E=E_w$. Then  $E$ does not have the JNP.
\end{theorem}

\begin{proof}
It is well known that the space $E$  is a dense linear subspace of $\IF^X$ for some set $X$ (for example, $X$ can be chosen to be a Hamel basis of $E'$). Hence we can identify the dual space $E'$ with the linear space of all functions $\mu:X\to \IF$ whose support $\supp(\mu):=\mu^{-1}(\IF\SM \{0\})$ is finite.

To see that $E$ does not have the JNP, we have to check that the identity map $E'_{w^\ast}\to E'_{\beta^\ast}$ is sequentially continuous. Fix any sequence $\M=\{\mu_n\}_{n\in\w}\subseteq E'$ that converges to zero in the topology $\sigma(E',E)$. Since $E$ is $c_0$-barrelled, the sequence $\M$ is equicontinuous, which means that $\M\subseteq U^\circ$ for some open neighborhood $U\subseteq E$ of zero. We can assume that $U$ is of the basic form:
\[
U=\{f\in E:\;|f(x)|<\e \; \mbox{ for all } \; x\in F \}
\]
for some finite set $F\subseteq X$ and some $\e>0$. The inclusion $\M\subseteq U^\circ$ implies that
$\bigcup_{\mu\in\M} \supp(\mu)\subseteq F$, and hence $\M$ is a subset of the finite-dimensional subspace $E'_F:=\{\mu\in E':\supp(\mu)\subseteq F\}$. Since any finite-dimensional locally convex space carries a unique locally convex topology, the sequence $\M\subseteq E'_F$ converges to zero also in the topology $\beta^\ast(E',E)$.
\end{proof}

A subset $B$ of a topological space $X$ is {\em functionally bounded} if for every continuous function $f:X\to\IR$ the image $f(B)$ is a bounded set in the real line.
By \cite{GK-DP}, for a Tychonoff space $X$, the function space $C_p(X)$ is $c_0$-barrelled if and only if every functionally bounded subset of $X$ is finite. Combining this characterization with Theorem~\ref{t:JNP-c0-barreled}, we obtain the following assertion.


\begin{proposition}
Let $X$ be a Tychonoff space such that every functionally bounded subset of $X$ is finite. Then the function space $C_p(X)$ does not have the JNP.
\end{proposition}

Let $X$ be a Tychonoff space.
Let $B_0(X):=C_p(X)$, and for every countable ordinal $\alpha\geq 1$, let $B_\alpha(X)$ be the family of all functions $f:X\to \mathbb{F}$ that are pointwise limits of sequences $\{f_n\}_{n\in\w}\subseteq \bigcup_{\beta<\alpha}B_\beta(X)$ in the Tychonoff product $\IF^X$.
All the spaces $B_\alpha(X)$ are endowed with the topology of pointwise convergence, inherited from the Tychonoff product $\IF^X$. The classes $B_\alpha(X)$ of Baire-$\alpha$ functions play an essential role in Functional Analysis, see for example \cite{BFT,Od-Ros}.

\begin{proposition} \label{c:JNP-Baire}
For every Tychonoff space $X$ and each nonzero countable ordinal $\alpha$, the function space $B_\alpha(X)$ does not have the JNP.
\end{proposition}

\begin{proof}
In  \cite{BG-Baire-LCS}, we proved that the space $B_\alpha(X)$ is barrelled. Since $B_\alpha(X)$ carries its weak topology, Theorem \ref{t:JNP-c0-barreled} applies.
\end{proof}

A locally convex space $E$ is {\em quasibarrelled} if and only if   the canonical map $\psi_E :E\to E''$ is a topological embedding.
The next assertion complements dually Proposition \ref{p:bar-weak-JNP}.

\begin{proposition} \label{p:JNP-dual}
Let $(E,\tau)$ be a quasibarrelled space, $\TTT$ be a locally convex topology on $E'$ compatible with the duality $(E,E')$,  and let $H:=(E',\TTT)$. Then:
\begin{enumerate}
\item[{\rm (i)}] $\beta^\ast(H',H)$ coincides with the topology $\tau$ on $E=H'$;
\item[{\rm (ii)}] the space $H$ has the JNP if and only if $E$ does not have the Schur property.
\end{enumerate}
\end{proposition}

\begin{proof}
(i) Since $H'=E$ and $\sigma(H',H)=\sigma(E,E')$, a subset $S$ of $H'=E$ is $\sigma(H',H)$-bounded if and only if  $S$ is bounded in $E$. Therefore a subset $U$ of $H$ is a barrel if and only if $U=S^\circ$ for some bounded subset $S\subseteq H'$. Hence a subset $B$ of $H$ is barrel-bounded if and only if  $B$ is a $\beta(E',E)$-bounded subset of $E'=H$. Therefore the topology $\beta^\ast(H',H)$ on $H'=E$ is the topology induced on $E$ from $E''$. Since $E$ is quasibarrelled, $E$ is a subspace of $E''$. Thus the topology $\beta^\ast(H',H)$ on $H'=E$ coincides with the original topology $\tau$ of the space $E$.
\smallskip

(ii) The JNP of $H$ means that the identity map
\[
\big( H',\sigma(H',H)\big)=\big( E,\sigma(E,E')\big) \to \big( H',\beta^\ast(H',H)\big)\stackrel{{\rm(i)}}{=} (E,\tau)
\]
is not sequentially continuous that, by definition, means that $E$ does not have the Schur property.
\end{proof}

Proposition \ref{p:JNP-dual} can be applied to one of the most important classes of locally convex spaces, namely, the class of free locally convex spaces introduced by Markov in \cite{Mar}. Recall that the {\em  free locally convex space}  $L(X)$ on a Tychonoff space $X$ is a pair consisting of a locally convex space $L(X)$ and  a continuous map $i: X\to L(X)$  such that every  continuous map $f$ from $X$ to a  locally convex space  $E$ gives rise to a unique continuous operator ${\bar f}: L(X) \to E$ such that $f={\bar f} \circ i$. The free locally convex space $L(X)$ always exists and is essentially unique.
We recall also that the set $X$ forms a Hamel basis of $L(X)$  and  the map $i$ is a topological embedding.
Various locally convex properties of free locally convex spaces are studied in \cite{Gab-Respected,Gabr-free-lcs}.

From the definition of $L(X)$ it easily follows the well known fact that the dual space $L(X)'$ of $L(X)$ is linearly isomorphic to the space $C(X)$. Indeed, the uniqueness of the operator $\bar f$ in the definition of the free locally convex space $(L(X),i)$ ensures that the operator $i':L(X)'\to C(X)$, $i':\mu\mapsto\mu\circ i$, is bijective and hence is a required linear isomorphism.  Via the pairing $(L(X)',L(X))=(C(X),L(X))$ we note that $C_p(X)'_{w^\ast} = L(X)_w$. Usually the space $L(X)_w$ is denoted by $L_p(X)$.

\begin{proposition} \label{p:L(X)-JNP}
For every Tychonoff space $X$, the space $L(X)$ does not have the JNP.
\end{proposition}

\begin{proof}
By Corollary 11.7.3 of \cite{Jar}, the space $E=C_p(X)$ is quasibarrelled for every Tychonoff space $X$. Since $C_p(X)$ carries its weak topology it is trivially has the Schur property. As we explained above the topology $\TTT$ of $L(X)$ is compatible with the duality $(E,E')$. Now (ii) of Proposition \ref{p:JNP-dual} applies.
\end{proof}

If $E$ is a locally convex space, we denote by $E'_\mu =\big(E', \mu(E',E)\big)$ the dual space $E'$ of $E$ endowed with the Mackey topology $\mu(E',E)$ of the dual pair $(E,E')$. The following statement is dual to Theorem \ref{t:BLV-theorem-JNP}.

\begin{theorem} \label{t:Frechet-dual-JNP}
For a Fr\'{e}chet space $E$, the following assertions are equivalent:
\begin{enumerate}
\item[{\rm (i)}] $E$ does not have the Schur property.
\item[{\rm (ii)}] $\big(E', \mu(E',E)\big)$ has the JNP.
\item[{\rm (iii)}] $\big(E', \sigma(E',E)\big)$ has the JNP. 
\item[{\rm (iv)}] $E$ has a bounded non-precompact sequence which does not have a subsequence equivalent to the unit basis of $\ell_1$.
\end{enumerate}
\end{theorem}

\begin{proof}
The clauses (i)--(iii) are equivalent by Proposition \ref{p:JNP-dual}, and Theorem 1.2 of \cite{Gabr-free-resp} exactly states that (i) and (iv)  are equivalent.
\end{proof}

We finish this section with the following ``hereditary'' result.

\begin{proposition} \label{p:JNP-hereditary}
Let an lcs $E$ have the JNP. Then
\begin{enumerate}
\item[{\rm (i)}] for every lcs $L$, the product $E\times L$ has the JNP;
\item[{\rm (ii)}] closed subspaces and Hausdorff quotients of an lcs with the JNP may fail to have the JNP;
\item[{\rm (iii)}] every lcs $H$ is topologically isomorphic to a closed subspace of an lcs with the JNP.
\end{enumerate}
\end{proposition}

\begin{proof}
(i) It is easy to check that $(E\times L)'_{\beta^{\ast}}=E'_{\beta^{\ast}}\times L'_{\beta^{\ast}}$. Since also $(E\times L)'_{w^{\ast}} =E'_{w^{\ast}}\times L'_{w^{\ast}}$ the JNP of $E$ implies that the identity map $(E\times L)'_{w^{\ast}}\to (E\times L)'_{\beta^{\ast}}$ is not continuous. Thus  $E\times L$ has the JNP.

(ii) Let $E$ be a Banach space and $L$ be a Fr\'{e}chet--Montel space. Then $L$ is topologically isomorphic to a closed subspace and to a  Hausdorff quotient of $E\times L$, and the assertion follows from (i) and Theorem \ref{t:BLV-theorem-JNP}.

(iii) If $Z$ is a Banach space, then $H$ embeds into $Z\times H$. It remains to note that,  by (i) and the Josefson--Nissenzweig theorem \ref{t:JN-theorem}, $Z\times H$ has the JNP.
\end{proof}


\section{The Josefson--Nissenzweig property in function spaces} \label{sec:JNP-function}


Let $X$ be a set, and let $f:X\to\IF$ be a function to the field $\IF$ of real or complex numbers. For a subset $A\subseteq X$ and $\e>0$, let
\[
\|f\|_A:=\sup(\{|f(x)|:x\in A\}\cup\{0\})\in [0,\infty].
\]
Observe that a subset $A\subseteq X$ is  {\em functionally bounded} iff $\|f\|_A<\infty$ for any continuous function $f:X\to\IR$.  A Tychonoff space $X$ is {\em pseudocompact} if $X$ is functionally bounded in $X$.

For a subfamily $\FF\subseteq\ff^X$, we put
\[
[A;\e]_{\FF}:=\{f\in \FF:\|f\|_A\le\e\}.
\]
If the family $\FF$ is clear from the context, then we shall omit the subscript $\FF$ and write $[A;\e]$ instead of $[A;\e]_{\FF}$.
A family $\mathcal S$ of subsets of  $X$ is {\em directed\/} if for any sets $A,B\in\mathcal S$ the union $A\cup B$ is contained in some set $C\in\mathcal S$.

For a Tychonoff space $X$, we denote by $C(X)$ the space of all continuous functions $f:X\to \IF$ on $X$ and let $C^b(X)$ be the subspace of $C(X)$ consisting of all bounded functions.

A Tychonoff space $X$ is defined to be a {\em $\mu$-space} if every functionally bounded subset of $X$ has compact closure in $X$. We denote by $\upsilon X$, $\mu X$ and $\beta X$ the Hewitt completion (=realcompactification), the Diedonn\'e completion and  the Stone-\v Cech compactification of $X$, respectively. It is known (\cite[8.5.8]{Eng}) that $X\subseteq \mu X\subseteq \upsilon X \subseteq \beta X$. Also it is known that all paracompact spaces and all realcompact spaces are Diedonn\'e complete and each Diedonn\'e complete space is a $\mu$-space, see \cite[8.5.13]{Eng}. On the other hand, each pseudocompact $\mu$-space is compact. By a {\em compactification} of a Tychonoff space $X$ we understand any compact Hausdorff space $\gamma X$ containing $X$ as a dense subspace.

For a Tychonoff space $X$, the space $C(X)$ carries many important {\em locally convex topologies}, i.e., topologies turning $C(X)$ into a locally convex space. For a locally convex topology $\TTT$ on $C(X)$, we denote by $C_\TTT(X)$ the space $C(X)$ endowed with the topology $\TTT$. The subspace $C^b(X)$ of $C_\TTT(X)$ with the induced topology is denoted by $C^b_\TTT(X)$.

Each directed family $\mathcal S$ of functionally bounded sets in a Tychonoff space $X$ induces a locally convex topology $\TTT_{\mathcal{S}}$ on $C(X)$ whose neighborhood base at zero consists of the sets $[S;\e]$ where $S\in\mathcal{S}$ and $\e>0$. The topology $\TTT_{\mathcal{S}}$ is called {\em the topology of uniform convergence on sets of the family $\mathcal{S}$}. The topology $\TTT_{\mathcal{S}}$ is Hausdorff if and only if the union $\bigcup\mathcal{S}$ is dense in $X$.

If $\mathcal{S}$
is the family of all finite, compact or functionally bounded subsets of $X$, respectively, then the topology $\TTT_{\mathcal{S}}$ will be denoted by $\TTT_p$, $\TTT_k$ or  $\TTT_b$, and the function space $C_{\TTT_{\mathcal{S}}}(X)$ will be denoted by $C_p(X)$,  $\CC(X)$ or $C_b(X)$, respectively.



Although the topologies $\TTT_p$, $\TTT_k$ and $\TTT_b$ are the most famous and well-studied there are other natural and important topologies on $C(X)$, for example, the topology $\TTT_s$ defined by the family 
of all finite unions of convergent sequences in $X$ and the topology $\TTT_c$ on $C(X)$ defined by the family 
of all {\em countable} functionally-bounded subsets of $X$. It is clear that
\[
\TTT_p\subseteq\TTT_s\subseteq\TTT_c\subseteq\TTT_b \;\;\mbox{ and } \;\; \TTT_p\subseteq\TTT_s\subseteq\TTT_k\subseteq\TTT_b.
\]



In \cite{BG-sGP}  we consider the following locally convex topologies on function spaces.
Let $X$ be a dense subspace of a Tychonoff space $M$ (for example, $M=\mu X$, $\upsilon X$ or $\beta X$). Then the union $\bigcup\mathcal{S}$ of the directed  family $\mathcal S$ of all  finite (resp. compact, functionally bounded) subsets of $X$ is dense in $M$. Therefore $\mathcal{S}$   defines the Hausdorff locally convex vector topology $\TTT_{\mathcal S}$} on the space $C(M)$ denoted by $\TTT_{p{\restriction}X}$ (resp. $\TTT_{k{\restriction}X}$, $\TTT_{b{\restriction}X}$). 

Numerous results concerning the JNP obtained below show that if a function space $C_\TTT(X)$ has the JNP,  then so do the spaces $C_\tau(X)$, $C_\TTT(\mu X)$, $C^b_\TTT(X)$, $C_\tau^b(X)$, $C_\TTT^b(\mu X)$, and $C(\beta X)$, where $\tau$ is a locally convex topology on $C(X)$, stronger than $\TTT$. These facts and the aforementioned discussion motivate the following definition which is very useful to unify all proofs.

\begin{definition} \label{def:between}
Let $I:E\to L$ be an injective continuous  operator between locally convex spaces $E$ and $L$. We shall say that a locally convex space $H$ is {\em between} the spaces $E$ and $L$ if there exist injective continuous operators $T_E:E\to H$ and $T_L:H\to L$ such that $T_L\circ T_E=I$.\qed
\end{definition}
For example, for any Tychonoff space $X$ the function space $\CC(X)$ is between $C_{b}(X)$ and $C_p(X)$.  Also for any compactification $\gamma X$ of $X$, the space $C_{k}^b(X)$ is between the spaces $C(\gamma X)$ and $\CC(X)$ linked by the restriction operator $I:C(\gamma X)\to \CC(X)$, $I:f\mapsto f{\restriction}_X$.

The {\em support} $\supp(\mu)$ of a linear functional $\mu:C(X)\to\IF$ is the set of all points $x\in X$ such that for every neighborhood $O_x\subseteq X$ of $x$  there exists a function $f\in C(X)$ such that $\mu(f)\ne 0$ and $\supp(f)\subseteq O_x$ where $\supp(f)=\overline{\{x\in X:f(x)\ne 0\}}$. The definition of $\supp(\mu)$ implies that it is a closed subset of $X$. We shall use the following assertions.

\begin{lemma}\label{l:3.1}
Let $X$ be a Tychonoff space, and let $\mathcal{S}$ be a directed family of functionally bounded sets in $X$. If a functional $\mu\in C(X)'$ is continuous in the topology $\TTT_{\mathcal S}$, then $\supp(\mu)\subseteq\overline{S}$ for some set $S\in\mathcal S$ such that $[S;0]\subseteq\mu^{-1}(0)$.
\end{lemma}

\begin{proof}
By the continuity of $\mu$ in the topology $\TTT_{\mathcal S}$, there exist a set $S\subseteq \mathcal S$ and $\e>0$  such that $\mu([S;\e])\subseteq(-1,1)$. Then
\[
[S;0]=\bigcap_{n\in\IN}[S;\tfrac\e{n}]\subseteq\bigcap_{n\in\w}\mu^{-1}\big((-\tfrac1n,\tfrac1n)\big)=\mu^{-1}(0).
\]
It remains to prove that $\supp(\mu)\subseteq\overline S$. In the opposite case we can find a function $f\in C(X)$ such that $\mu(f)\ne 0$ and $\supp(f)\cap \overline S=\emptyset$. On the other hand, $f\in [S;0]\subseteq \mu^{-1}(0)$ and hence $\mu(f)=0$. This contradiction shows that $\supp(\mu)\subseteq\overline S$.
\end{proof}

\begin{lemma} \label{l:3.2}
Let $X$ be a Tychonoff space. If a linear functional $\mu\in C(X)'$ is continuous in the topology $\TTT_{k}$, then $\supp(\mu)$ is a compact subset of $X$ and $[\supp(\mu);0]\subseteq\mu^{-1}(0)$.
\end{lemma}

\begin{proof}
By the continuity of $\mu$ in the topology $\TTT_{k}$, there exist a compact subset $K\subseteq {X}$ and $\e>0$ such that $\mu([K;\e])\subseteq(-1,1)$. By (the proof of) Lemma~\ref{l:3.1}, $\supp(\mu)\subseteq K$ and $[K;0]\subseteq\mu^{-1}(0)$. Since $\supp(\mu)$ is a closed subset of $X$, $\supp(\mu)$ is closed in $K$ and hence $\supp(\mu)$ is a compact subset of $X$.

It remains to prove that $[\supp(\mu);0]\subseteq\mu^{-1}(0)$. To derive a contradiction, assume that $[\supp(\mu);0]\not\subseteq\mu^{-1}(0)$ and hence there exists a continuous function $f\in C(X)$ such that $\mu(f)\ne 0$ but $f{\restriction}_{\supp(\mu)}=0$. Multiplying $f$ by a suitable constant, we can assume that $\mu(f)=2$. Embed the space $X$ into its Stone--\v{C}ech compactification $\beta X$. By the Tietze--Urysohn Theorem, there exists a continuous function $\bar f\in C(\beta X)$ such that $\bar f{\restriction}_K=f{\restriction}_K$. It follows from $[K;0]\subseteq\mu^{-1}(0)$ that $\mu(f)=\mu(\bar f{\restriction}_X)$.

Consider the open neighborhood $U=\{x\in \beta X:|\bar f(x)|<\e\}$ of $\supp(\mu)$ in $\beta X$. By the definition of support $\supp(\mu)$, every point $x\in K\setminus U$ has an open neighborhood $O_x\subseteq \beta X$ such that $\mu(g)=0$ for any function $g\in C(X)$ with $\supp(g)\subseteq O_x\cap X$. Observe that $U\cup\bigcup_{x\in K\setminus U}O_x$ is an open neighborhood of the compact set $K$ in $\beta X$.
So there is a finite family $F\subseteq K\SM U$ such that $K\subseteq U\cup\bigcup_{x\in F} O_x$. Let $1_{\beta X}$ denote the constant function $\beta X\to\{1\}$. By the paracompactness of the compact space $\beta X$, there is a finite family $\{\lambda_0,\dots,\lambda_n\}$ of continuous functions $\lambda_i:\beta X \to[0,1]$ such that $\sum_{i=0}^n\lambda_i=1_{\beta X}$ and for every $i\in\{0,\dots,n\}$, the support $\supp(\lambda_i)$ is contained in some set $V\in\{\beta X\setminus K,U\}\cup\{O_x:x\in F\}$. We lose no generality assuming that
\[
\bigcup_{i=0}^j \supp(\lambda_i)\subseteq \beta X\setminus K, \quad \bigcup_{i=j+1}^s\supp(\lambda_i)\subseteq U,
\]
and for every $i\in\{s+1,\dots,n\}$ there exists $x_i\in F$ such that $\supp(\lambda_i)\subseteq O_{x_i}$.

Replacing the functions, $\lambda_0,\dots,\lambda_j$ by the single function $\sum_{i=0}^j\lambda_i$ and the functions $\lambda_{j+1},\dots,\lambda_s$ by the single function $\sum_{j+1}^s\lambda_i$, we can assume that $j=0$ and $s=1$. In this simplified case we have $\supp(\lambda_0)\subseteq \beta X\setminus K$, $\supp(\lambda_1)\subseteq U$ and $\supp(\lambda_i)\subseteq O_{x_i}$ for all $i\in\{2,\dots,n\}$.

For every $i\in n$, consider the function $f_i\in C(\beta X)$, defined by  $f_i(x):=\lambda_i(x)\cdot \bar f(x)$ for $x\in\beta X$. Then $\bar f=\sum_{i\in n} f_i$.

It follows from $\supp(f_0)\subseteq\supp(\lambda_0)\subseteq\beta X\setminus K$ and $[K;0]\subseteq \mu^{-1}(0)$ that  $f_0{\restriction}_K =0$ and $\mu(f_0{\restriction}_X)=0$.

Since $K\subseteq U$, $\bar f(U)\subseteq(-\e,\e)$ and $\supp(f_1)\subseteq \supp(\lambda_1)\subseteq U$, the function $f_1{\restriction}_X= ({\bar f}{\cdot}\lambda_1){\restriction}_X$ belongs to the set $[K;\e]\subseteq\mu^{-1}\big((-1,1)\big)$ and hence
\begin{equation}\label{eq:fle1}
\big|\mu\big(f_1{\restriction}_X\big)\big|\le 1.
\end{equation}

For every $i\in\{2,\dots,n\}$, we have $\supp(f_i)\subseteq\supp(\lambda_i)\subseteq O_{x_i}$ and hence $\mu(f_i{\restriction}_X)=0$ by the choice of $O_{x_i}$.
\smallskip

Now we see that
\[
2=\mu(f)=\mu(\bar f{\restriction}_X)=\mu(f_0{\restriction}_X)+\mu(f_1{\restriction}_X) +\mu\Big(\sum_{i=2}^n f_i{\restriction}_X\Big)=\mu(f_1{\restriction}_X),
\]
which contradicts (\ref{eq:fle1}).
\end{proof}

\begin{lemma}  \label{l:3.3}
Let $X$ be a Tychonoff space. 
For any bounded subset $\mathcal M\subseteq (C_{b}(X))'_{w^\ast}$ the set $\supp(\M)=\bigcup_{\mu\in\mathcal M}\supp(\mu)$ is functionally bounded in $X$.
\end{lemma}

\begin{proof}
To derive a contradiction, assume that the set $\supp(\M)$ is not functionally bounded in $X$. Then there exists a continuous function $\varphi:X\to[0,\infty)$ such that the set $\varphi(\supp(\M))$ is not bounded in the real line. Inductively we shall choose sequences of functionals $\{\mu_n\}_{n\in\w}\subseteq\M$, points $\{x_n\}_{n\in\w}\subseteq\supp(\M)$ and functionally bounded sets $\{S_n\}_{n\in\w}$ in $X$ such that for every $n\in\w$ the following conditions are satisfied:
\begin{enumerate}
\item $\varphi(x_n)>3+\sup\varphi(\bigcup_{i<n}S_i)$;
\item $x_n\in\supp(\mu_n)\subseteq \overline{S_n}$ and $[S_n;0]\subseteq\mu^{-1}(0)$.
\end{enumerate}

To start the inductive construction, choose any point $x_0\in\supp(\M)$ with $\varphi(x_0)>3$ and find a functional $\mu_0\in\M$ such that $x_0\in\supp(\mu_0)$. By Lemma~\ref{l:3.1}, there exists a functionally bounded set  $S_0\subseteq X$ such that $\supp(\mu_0)\subseteq\overline{S_0}$ and $[S_0,0]\subseteq\mu_0^{-1}(0)$. Assume that for some $n\in\IN$ we have chosen functionally bounded sets $S_0,\dots,S_{n-1}$ in $X$ satisfying the inductive conditions (1) and (2). 
As the set $\varphi(\supp(\M))$ is unbounded and the set $\varphi(\bigcup_{i<n}S_i)$ is bounded in the real line, we can find a point $x_n\in\supp(\M)$ satisfying the inductive condition (1). Since $x_n\in\supp(\M)=\bigcup_{\mu\in\M}\supp(\mu)$ we can find a functional $\mu_n\in\M$ such that $x_n\in\supp(\mu_n)$. By Lemma~\ref{l:3.1}, for the functional $\mu_n\in C_b(X)'$ there exists a functionally bounded set $S_n\subseteq X$ satisfying the inductive condition (2). This completes the inductive step.
\smallskip

Now, for every $n\in\w$, consider the open neighborhood $O_n=\{x\in X:|\varphi(x)-\varphi(x_n)|<1\}$ of the point $x_n$. The inductive condition (1) ensures that $\varphi(x_n)-\varphi(x_i)>3$ for any $i<n$, which implies that the family $(O_n)_{n\in\w}$ is discrete in $X$. Moreover, $O_m\cap S_n=\emptyset$ for any numbers $n<m$.

For every $n\in \w$,  the definition of the support $\supp(\mu_n)\ni x_n$ implies the existence of a function $f_n\in C(X)$ such that $\mu_n(f_n)\ne 0$ and $\supp(f_n)\subseteq O_n$. Multiplying $f_n$ by a suitable constant, we can assume that
\begin{equation}\label{eq:choice-fn}
\mu_n(f_n)>n+\sum_{i<n}|\mu_n(f_i)|.
\end{equation}
Since the family $(O_n)_{n\in\w}$ is discrete, so is the family $(\supp(f_n))_{n\in\w}$. Consequently, the function $f=\sum_{n\in\w}f_n:X\to\IF$ is well-defined and continuous.

For every numbers $n<m$ we have $\supp(f_m)\cap S_n\subseteq O_m\cap S_n=\emptyset$ and hence $f{\restriction}_{S_n}=\sum_{i\le n}f_i{\restriction}_{S_n}$. Since $[S_n;0]\subseteq\mu_n^{-1}(0)$, we have
$$\mu_n(f)=\sum_{i\le n}\mu_n(f_i)\ge \mu_n(f_n)-\sum_{i<n}|\mu_n(f_i)|>n$$
according to (\ref{eq:choice-fn}).
Consequently, $\sup\{|\mu(f)|:\mu\in\M\}\ge\sup_{n\in\w}|\mu_n(f)|=\infty$, which contradicts the boundedness of the set $\M$ in $C(X)'_{w^\ast}$.
\end{proof}

The next theorem shows that the JNP has some ``hereditary'' type property with respect to finer locally convex topologies.

\begin{theorem} \label{t:JNP-between}
Let $\gamma X$ be a compactification of a Tychonoff space $X$. Let $Y$ be an lcs between $\CC(X)$ and $C_p(X)$, and let $Z$ be an lcs between $C(\gamma X)$ and $Y$. If the lcs $Y$ has the JNP, then $Z$ has  the JNP, too.
\end{theorem}

\begin{proof}
We can identify the lcs $Y$ with the function space $C(X)$ endowed with a suitable locally convex topology $\TTT$ such that $\TTT_p\subseteq \TTT\subseteq \TTT_k$. 
Since $Z$ is between $C(\gamma X)$ and $Y$, there exist injective operators $T_\gamma:C(\gamma X)\to Z$ and $T:Z\to Y$ such that $T\circ T_\gamma(f)=f{\restriction}_X$ for every $f\in C(\gamma X)$.

Assuming that the lcs $Y$ has the JNP, find a null sequence $\{\mu_n\}_{n\in\w}\subseteq Y'_{w^\ast}$ such that $\inf_{n\in\w}\|\mu_n\|_B>\e$ for some barrel-bounded set $B\subseteq Y$ and some $\e>0$. For every $n\in\w$, choose an element $f_n\in B$ such that $|\mu_n(f_n)|>\e$. Since $\TTT\subseteq\TTT_k\subseteq\TTT_b$, the linear functionals $\mu_n\in Y'$ are continuous in the topology $\TTT_b$. The continuity of the identity operator $C_b(X)\to Y$ implies that the sequence $(\mu_n)_{n\in\w}$ is weak$^\ast$ null in $C_b(X)'$.  By Lemma~\ref{l:3.3}, the set $S=\bigcup_{n\in\w}\supp(\mu_n)$ is functionally bounded in $X$.
It follows that the set $[S;1]$ is a barrel in $C_p(X)$ and hence in $Y$. Since  $B$ is barrel-bounded in $Y$, there exists $r\in \IN$ such that $B\subseteq r\cdot [S;1]$.

For every $n\in\w$, the functional $\mu_n$ is continuous in the topology $\TTT\subseteq\TTT_k$ and hence,  by Lemma~\ref{l:3.2}, $\mu_n$ has compact support $\supp(\mu_n)$. By the Tietze--Urysohn Theorem, there exists a continuous function $g_n\in C(\gamma X)$ such that $g_n{\restriction}_{\supp(\mu_n)}=f_n{\restriction}_{\supp(\mu_n)}$ and $\|g_n\|_{\gamma X}=\|f_n\|_{\supp(\mu_n)}$  and hence, by Lemma \ref{l:3.2}, $\mu_n\big(g_n{\restriction}_X\big)=\mu_n(f_n)$.

The definition of the number $r$ guarantees that
\[
\sup_{n\in\w}\|g_n\|_{\gamma X}=\sup_{n\in\w}\|f_n\|_{\supp(\mu_n)}\le \sup_{f\in B}\|f\|_S\le r.
\]
This means that the set $\{g_n\}_{n\in \w}$ is bounded in the Banach space $C(\gamma X)$. Since Banach spaces are barrelled, $\{g_n\}_{n\in\w}$ is barrel-bounded in $C(\gamma X)$. Now the continuity of the operator $T_\gamma:C(\gamma X)\to Z$ ensures that the set $D=\{T_\gamma(g_n)\}_{n\in\w}$ is barrel-bounded in $Z$. For every $n\in\w$, consider the linear continuous functional $\lambda_n=\mu_n\circ T\in Z'$.
The convergence $\mu_n\to  0$ in $Y'_{w^\ast}$ implies the convergence $\lambda_n\to 0$ in $Z'_{w^\ast}$.

Finally, observe that for every $n\in\w$ we have
\[
\lambda_n(T_\gamma(g_n))=\mu_n(T\circ T_\gamma(g_n))=\mu_n(g_n{\restriction}X)=\mu_n(f_n)
\]
 and hence
$
\|\lambda_n\|_D\ge|\mu_n(f_n)|>\e
$
 and $\|\lambda_n\|_D\not\to 0$, witnessing that the lcs $Z$ has the JNP.
\end{proof}
\smallskip

Applying Theorem \ref{t:JNP-between} to $\gamma X=\beta X$, $Y=C_p(X)$ and $Z=\CC(X)$ we obtain

\begin{corollary} \label{c:JNP-Cp-Ck}
Let $X$ be a Tychonoff space such that the space $C_p(X)$ has the JNP. Then also the space $\CC(X)$ has the JNP. Moreover, the weak$^\ast$ null sequence $\{\mu_n\}_{n\in\w}$ in the dual space $\CC(X)'$ witnessing the JNP of $\CC(X)$ can be chosen such that all $\mu_n$ have finite support.
\end{corollary}
\smallskip

\begin{corollary} \label{c:JNP-Ck-C-gamma}
Let $X$ be a Tychonoff space such that the space $\CC(X)$ has the JNP. Then:
\begin{enumerate}
\item[{\rm(i)}]  the spaces  $\CC(\mu X)$, $\CC(\upsilon X)$, $C(\beta X)$ and $C_b(X)$ have the JNP;
\item[{\rm(ii)}]  the spaces  $\CC^b(\mu X)$, $\CC^b(\upsilon X)$, and $C_b^b(X)$ have the JNP.
\end{enumerate}
Moreover, the weak$^\ast$ null sequence $\{\mu_n\}_{n\in\w}$ in their dual spaces witnessing the JNP can be chosen such that all $\mu_n$ have compact support contained in $X$.
\end{corollary}

\begin{proof}
The corollary follows from Theorem \ref{t:JNP-between} applied to  $\gamma X=\beta X$, $Y=\CC(X)$ and $Z$ is one of the spaces from (i) and (ii).
\end{proof}

If $E$ is a locally convex space, we denote  by $E_\beta$ the space $E$ endowed with the locally convex topology $\beta(E,E')$ whose neighborhood base at zero consists of barrels.
To characterize function spaces $C_p(X)$ and $\CC(X)$ having the JNP we shall use the following proposition.

\begin{proposition} \label{p:bb-Ck}
Let $X$ be a Tychonoff space, and let $\TTT$ be a locally convex topology on $C(X)$ such that $\TTT_p\subseteq \TTT\subseteq \TTT_k$. Then:
\begin{enumerate}
\item[{\rm(i)}] for every barrel $D$ in $C_\TTT(X)$, there are a functionally bounded subset $A$ of $X$ and $\e>0$ such that $[A;\e]\subseteq D$.
\item[{\rm(ii)}] a subset $\F\subseteq C_\TTT(X)$  is barrel-bounded if and only if for any functionally bounded set $A\subseteq X$, the set $\F(A):=\bigcup_{f\in\F}f(A)$ is bounded in $\IF$;
\item[{\rm(iii)}] $\big(C_\TTT(X)\big)_{\!\beta} = C_b(X)$.
\end{enumerate}
\end{proposition}

\begin{proof}
(i) Let $\upsilon X$ be the realcompactification of $X$, and let $R:\CC(\upsilon X)\to \CC(X)$, $R:f\mapsto f{\restriction}_X$, be the restriction operator. Since every continuous function $f:X\to \ff$ admits a unique continuous extension to $\upsilon X$, the operator $R$ is bijective. As $\upsilon X$ is a $\mu$-space, the Nachbin--Shirota theorem \cite[11.7.5]{Jar} implies that $\CC(\upsilon X)$ is barrelled. The continuity of the operators $\CC(\upsilon X)\to \CC(X)\to C_\TTT(X)$ implies that for every barrel $D$ in $C_\TTT(X)$, the preimage $R^{-1}(D)$ is a barrel in $\CC(\upsilon X)$. Since the space $\CC(\upsilon X)$ is barrelled, $R^{-1}(D)$ is a neighborhood of zero. So, there exists a compact subset $K\subseteq\upsilon X$ and $\e>0$ such that $[K;\e]\subseteq R^{-1}(D)$. It follows that the set $A=K\cap X$ is functionally bounded and closed in $X$. It is clear that the set $[A;\e]$ is a barrel in $C_p(X)$ and hence also in $C_\TTT(X)$. We claim that $[A;\e]\subseteq D$. Indeed, suppose for a contradiction that $[A;\e]\setminus D$ contains some function $f$.  Since the barrel $D$ is closed in $\TTT$ and the identity operator $\CC(X)\to C_\TTT(X)$ is continuous, there exist $\delta>0$ and a compact set $C\subseteq X$ such that $(f+[C;\delta])\cap D=\emptyset$. If follows from $f\in [A;\e]$ that $f[C\cap K]=f[C\cap A]\subseteq[-\e,\e]$. By the Tietze--Urysohn Theorem, there exists a continuous function $g:K\to [-\e,\e]$ such that $g(x)=f(x)$ for every $x\in K\cap C$. Define the function $\varphi:C\cup K\to\ff$ by the formula
\[
\varphi(x)=\begin{cases}
f(x)&\mbox{if $x\in C$};\\
g(x)&\mbox{if $x\in K$};
\end{cases}
\]
and observe that it is well-defined and continuous. By the Tietze--Urysohn Theorem, the function $\varphi$ can be extended to a bounded continuous function $\psi:\upsilon X\to\ff$. Then $\psi{\restriction}_X\in (f+[C;\delta])\cap R([K;\e])\subseteq (f+[C;\delta])\cap D$, which contradicts the choice of $C$ and $\delta$. This contradiction shows that $B\subseteq D$.
\smallskip

(ii) To prove the ``only if'' part, take a functionally bounded set $A\subseteq X$ and observe that the set $[A;1]$ is a barrel in $C_p(X)$, and hence $B$ is a barrel in $C_\TTT(X)$. If $\F$ is barrel-bounded, then $\F\subseteq n\cdot [A;1]$ for some  $n\in\w$, which implies $\F(A)\subseteq \{z\in\IF:|z|\le n\}$.
\smallskip

To prove the ``if'' part, assume that for any functionally bounded set $A\subseteq X$, the set $\F(A)$ is bounded in $\IF$. To see that $\F$ is barrel-bounded in $C_\TTT(X)$, fix any barrel $D\subseteq C_\TTT(X)$. By (i), there are a  functionally bounded set $Z$ in $X$  and $\e>0$ such that $[Z;\e]\subseteq D$. It is clear that $[Z;\e]$ is a barrel in $C_p(X)$ and hence in $C_\TTT(X)$. By our assumption, the set $\F(Z)$ is bounded, and hence the real number $r=\sup\{|f(x)|:f\in\F,\;x\in Z\}$ is well-defined. It is clear that $\F \subseteq [Z;r]=\tfrac{r}{\e} [Z;\e]\subseteq\tfrac{r}{\e}D$. Thus $\F$ is barrel-bounded.
\smallskip

(iii) Set $E=C_\TTT(X)$. By (i), the family $\mathcal{B}$ of sets of the form $[A;\e]$, where $A$ is functionally bounded in $X$ and $\e>0$, is a neighborhood base at zero of the topology $\beta(E,E')$. On the other hand, by the definition of $\TTT_b$, the family  $\mathcal{B}$  is also  a neighborhood base at zero of the topology $\TTT_b$. Thus $\beta(E,E')=\TTT_b$, as desired.
\end{proof}

In the following theorem for a Tychonoff space $X$ and a functional $\mu\in C(X)'$ we put
\[
\|\mu\|:=\|\mu\|_{[X;1]}.
\]

\begin{theorem}  \label{t:JNP-C-norm}
Let $X$ be a Tychonoff space, and let $\TTT$ be a locally convex topology on $C(X)$ such that $\TTT_p\subseteq\TTT\subseteq\TTT_k$. The function space $C_{\TTT}(X)$ has the JNP if and only if there is a null sequence $\{\mu_n\}_{n\in\w}\subseteq C_{\TTT}(X)'_{w^\ast}$ such that $\|\mu_n\|=1$  for every $n\in\w$.
\end{theorem}

\begin{proof}
To prove the ``if'' part, assume that the weak dual $E'_{w^\ast}$ of the locally convex space $E=C_{\TTT}(X)$ contains a null sequence $\{\mu_n\}_{n\in\w}\subseteq E'_{w^\ast}$ such that $\|\mu_n\|=1$ for all $n\in\w$.  To see that $E$ has the Josefson--Nissenzweig property, it suffices to show that the sequence $\{\mu_n\}_{n\in\w}$ diverges in the topology $\beta^\ast(E',E)$.

Since the null sequence $\M=\{\mu_n\}_{n\in\w}\subseteq E'_{w^\ast}$ is bounded, we can apply Lemma \ref{l:3.3} and conclude that the set $\supp(\M)$ is
functionally bounded in $X$. By (ii) of Proposition \ref{p:bb-Ck}, the set $[X;1]$ is barrel-bounded in $E$. Since $\|\mu_n\|_{[X;1]}=\|\mu_n\|=1\not\to 0$,  the sequence $\{\mu_n\}_{n\in\w}$ does not converge to zero in the topology $\beta^\ast(E',E)$. Thus the identity map $E'_{w^\ast}\to E'_{\beta^\ast}$ is not sequentially continuous and $E=C_\TTT(X)$ has the JNP.
\smallskip

To prove the ``only if'' part, assume that the space $E=C_\TTT(X)$ has the Josefson--Nissenzweig property and hence there is a sequence $\M=\{\mu_n\}\subseteq E'_{w^\ast}$ that converges to zero in the topology $\sigma(E',E)$ but not in the topology $\beta^\ast(E',E)$. By Lemma  \ref{l:3.3}, the set $\supp(\M)$ is functionally bounded in $X$.

We claim that $\|\mu_n\|\not\to 0$. Indeed, suppose for a contradiction that $\lim_{n}\|\mu_n\|=0$. In this case we shall prove that the sequence $\{\mu_n\}_{n\in\w}$ converges to zero in the topology $\beta^\ast(E',E)$. Given any barrel-bounded set $ B\subseteq E$, we need to find an $n\in\w$ such that $\|\mu_k\|_B\le 1$ for every $k\ge n$.  By (ii) of Proposition~\ref{p:bb-Ck}, the number
\[
r:=\sup\{|f(x)|:f\in B,\;x\in\supp(\M)\}
\]
is finite. Since $\|\mu_n\|\to 0$, there exists a number $m\in\w$ such that $r\cdot \|\mu_k\|\le1$ for all $k\ge m$. Fix $k\geq m$ and choose an arbitrary  $f\in B$. Since $\TTT\subseteq \TTT_k$, the functional $\mu_k$ has  compact support $\supp(\mu_n)$, according to Lemma~\ref{l:3.2}. By the Tietze--Urysohn Theorem, there is an extension ${\hat f}\in C(X)$ of the function $f{\restriction}_{\supp(\mu_n)}$ such that $\|{\hat f}\|_{X}=\|f\|_{\supp(\mu_n)}\leq r$. By Lemma \ref{l:3.2}, we have $\mu_k({\hat f})=\mu_k(f)$. Therefore
\[
|\mu_k(f)|= \big|\mu_k({\hat f})\big|\leq \|\mu_k\|\cdot \|{\hat f}\|_{X}\le \|\mu_k\|\cdot r\le 1,
\]
which implies that $\|\mu_k\|_B\le 1$. Therefore, the sequence $\{\mu_k\}_{k\in\w}$ converges to zero in the topology $\beta^\ast(E',E)$, which is a desired contradiction.

By the claim, there is  $\e>0$ such that the set $\Omega=\{n\in\w:\|\mu_n\|\ge\e\}$ is infinite.
Write $\Omega$ as $\Omega=\{n_k\}_{k\in\w}$,  where $n_0<n_1<\cdots$, and observe that the sequence $\{\eta_k\}_{k\in\w}$ of functionals
\[
\eta_k :=\frac{\mu_{n_k}}{\|\mu_{n_k}\|}
\]
converges to zero in the topology $\sigma(E',E)$ and consists of functionals of norm 1.
\end{proof}

The next corollary of Theorem~\ref{t:JNP-C-norm} and Lemma~\ref{l:3.1} shows that for $C_p$-spaces the Josefson--Nissenzweig property introduced in Definition~\ref{def:JNP} is equivalent to the JNP introduced in \cite{BKS}.

\begin{corollary}  \label{c:JNP-Cp}
For a Tychonoff space $X$, the function space $C_p(X)$ has the JNP if and only if there is a null sequence $\{\mu_n\}_{n\in\w}\subseteq C_p(X)'_{w^\ast}$ that consists of finitely supported sign-measures of norm 1.
\end{corollary}

Applying Theorem \ref{t:JNP-C-norm} and Lemma~\ref{l:3.1}  to the compact-open topology $\TTT=\TTT_k$ on $C(X)$, we obtain

\begin{corollary}  \label{c:JNP-Ck}
For a Tychonoff space $X$, the function space $\CC(X)$ has the JNP if and only if there is a null sequence $\{\mu_n\}_{n\in\w}\subseteq \CC(X)'_{w^\ast}$ that consists of compactly supported sign-measures of norm 1.
\end{corollary}

\begin{corollary}\label{c:seq-JNP}
If a Tychonoff space $X$ contains a non-trivial convergent sequence, then the function space $C_p(X)$ has the JNP.
\end{corollary}

\begin{proof}
Let $\{x_n\}_{n\in\w}\subseteq X$ be a non-trivial sequence that converges to some point $x\in X\setminus\{x_n\}_{n\in\w}$. For every $n\in\w$, consider the functional $\chi_n\in C_p(X)'$ defined by $\chi_n(f)=\frac12\big(f(x_n)-f(x)\big)$ for any $f\in C_p(X)$. It follows that $\{\chi_n\}_{n\in\w}$ is a null sequence in $C_p(X)'_{w^\ast}$ with $\|\chi_n\|=1$ for all $n\in\w$. By Corollary~\ref{c:JNP-Cp}, the function space $C_p(X)$ has the JNP.
\end{proof}

For function spaces over pseudocompact spaces, the JNP has even better ``hereditary'' properties.

\begin{theorem} \label{t:JNP-hereditary-1}
Let $X$ be a pseudocompact space, $D$ be a dense subset in $X$, and let $\tau$ and $\TTT$ be two locally convex topologies on $C(X)$ such that $\TTT_{p{\restriction}D}\subseteq\tau\subseteq\TTT\subseteq\TTT_{b}$. If the space $C_\tau(X)$ has JNP, then also the space $C_\TTT(X)$ has the JNP.
\end{theorem}

\begin{proof}
If the lcs $C_\tau(X)$ has the JNP, then there exists a null sequence $\{\mu_n\}_{n\in\w}\subseteq C_\tau(X)'_{w^\ast}$ such that $\|\mu_n\|_B\not\to 0$ for some barrel-bounded set $B\subseteq C_\tau(X)$. Since $\tau\subseteq\TTT$, the functionals $\mu_n$ are $\TTT$-continuous and hence $\{\mu_n\}_{n\in\w}$ is a null sequence in $C_\TTT(X)'_{w^\ast}$. We claim that the set $B$ remains barrel-bounded in $C_\TTT(X)$. Indeed, since $\TTT\subseteq\TTT_b$, any barrel $A$ in $C_\TTT(X)$ is a barrel in $C_b(X)$. Since $X$ is pseudocompact, $C_b(X)$ is a Banach space whose topology is generated by the norm $\|\cdot\|_X$. Since Banach spaces are barrelled, the barrel $A$ is a neighborhood of zero and hence $[X;\e]\subseteq A$ for some $\e>0$. As $D$ is dense in $X$, the barrel $[X;\e]$ is a barrel in $C_{p{\restriction}D}(X)$ and hence a barrel in $C_\tau(X)$. Since the set $B$ is barrel-bounded in $C_\tau(X)$, there exists a positive real number $r$ such that $B\subseteq r\cdot [X;\e]\subseteq r\cdot A$, which means that $B$ is barrel-bounded in $C_\TTT(X)$. Since $\|\mu_n\|_B\not\to 0$, the locally convex space $C_\TTT(X)$ has the JNP.
\end{proof}

Recall that a Tychonoff space $X$ is called an {\em $F$-space} if every functionally open set $A$ in $X$ is {\em $C^\ast$-embedded} in the sense that every bounded continuous function $f:A\to \IR$ has a continuous extension $\bar f:X\to\IR$. For  numerous equivalent conditions for a Tychonoff space $X$ being an $F$-space, see \cite[14.25]{GiJ}. In particular, the Stone--\v{C}ech compactification $\beta \Gamma$ of a discrete space $\Gamma$ is a compact $F$-space. The following example generalizes the example given in (2) after Definition 1 of \cite{BKS} with a more detailed proof.

\begin{example} \label{exa:F-space-JNP}
For any infinite compact $F$-space $K$, the function space $C_p(K)$ does not have the JNP.
\end{example}

\begin{proof}
Let $S=\{\mu_n\}_{n\in\w}\subseteq C_p(K)'_{w^\ast}$ be a null sequence. Then $S$ is a weak$^\ast$ null sequence in the dual space $M(K)$ of all regular Borel sign-measures on $K$ of the Banach space $C(K)$. By  Corollary 4.5.9 of \cite{Dales-Lau},  the Banach space $C(K)$ has the Grothendieck property, which implies that $\mu_n \to 0$ in the weak topology of the Banach space $M(K)$. But since $S$ contains only sign-measures with finite support, it is contained in the Banach subspace $\ell_1(K)$ of $M(K)$. Now the Schur property of $\ell_1(K)$ implies that $S$ converges to zero in $\ell_1(K)$, i.e. $\| \mu_n\|\to 0$.  By Corollary \ref{c:JNP-Cp}, the function space $C_p(K)$ does not have the JNP.
\end{proof}
By Corollary \ref{c:JNP-Cp-Ck}, if the space $C_p(X)$ has the JNP, then also the space $\CC(X)$ has the JNP. However, the converse is not true in general as Example \ref{exa:F-space-JNP} shows. So the JNP is not equivalent for $C_p(X)$ and $\CC(X)$.


\end{document}